\documentclass[12pt]{amsart}
\usepackage{SSdefn}
\usepackage{SSams}
\usepackage{scrextend}

\usepackage{eucal}
\usepackage{tikz}

\newcommand{\nc}{\newcommand}
\nc{\dmo}{\DeclareMathOperator}

\dmo{\Frac}{Frac}
\dmo{\Proj}{Proj}
\dmo{\Res}{Res}
\dmo{\Coh}{Coh}
\dmo{\QCoh}{QCoh}
\dmo{\Inj}{Inj}
\dmo{\uIsom}{\underline{Isom}}
\dmo{\uHom}{\underline{Hom}}

\nc{\gl}{\mathfrak{gl}}
\nc{\gen}{\mathrm{gen}}
\nc{\el}{\mathrm{el}}
\nc{\vis}{\mathrm{vis}}
\nc{\bone}{\mathbf{1}}
\nc{\PreMod}{\mathrm{PreMod}}
\nc{\pl}{\mathrm{pl}}
\nc{\rep}{\mathrm{rep}}
\nc{\supprk}{\operatorname{supp.rk.}}
\nc\ylw{\Yfillcolour{black}}
\nc\wht{\Yfillcolour{white}}
\nc{\coloneq}{\mathrel{\mathop:}\mkern-1.2mu=}
\nc{\abs}[1]{|#1|}
\nc{\inn}{\mathrm{in}}
\nc{\FI}{\mathbf{FI}}
\nc{\FS}{\mathbf{FS}}

\dmo{\FP}{\cC}
\dmo{\wA}{\Sym( \bC^\infty \otimes \bC^\infty)}
\dmo{\ssA}{\Sym(\Sym^2 \bC^\infty)}
\dmo{\swA}{\Sym(\bigwedge^2 \bC^\infty)}
\dmo{\wsA}{\bigwedge(\Sym^2 \bC^\infty)}
\dmo{\wwA}{\bigwedge(\bigwedge^2 \bC^\infty)}
\dmo{\allA}{A}
\dmo{\ssub}{\mathfrak{ssub}}
\dmo{\swub}{\mathfrak{swub}}
\dmo{\wub}{\mathfrak{wub}}
\dmo{\wsub}{\mathfrak{wsub}}
\dmo{\wwub}{\mathfrak{wwub}}
\nc{\ML}[2]{P_{[#1]}([#2])}
\nc{\MS}[1]{P_{[#1]}}
\nc{\shift}[1]{\mathcal{S}_{+#1}}
\nc{\symm}[1]{\mathfrak{S}_{#1}}
\nc{\spcht}[1]{\bM_{#1}}

\author{Rohit Nagpal}
\address{Department of Mathematics, University of Wisconsin, Madison, WI}
\curraddr{Department of Mathematics, The University of Chicago, Chicago, IL}
\email{\href{mailto:nagpal@math.uchicago.edu}{nagpal@math.uchicago.edu}}
\urladdr{\url{http://math.uchicago.edu/~nagpal/}}

\author{Steven V Sam}
\address{Department of Mathematics, University of California, Berkeley, CA}
\curraddr{Department of Mathematics, University of Wisconsin, Madison, WI}
\email{\href{mailto:svs@math.wisc.edu}{svs@math.wisc.edu}}
\urladdr{\url{http://math.wisc.edu/~svs/}}

\author{Andrew Snowden}
\address{Department of Mathematics, University of Michigan, Ann Arbor, MI}
\email{\href{mailto:asnowden@umich.edu}{asnowden@umich.edu}}
\urladdr{\url{http://www-personal.umich.edu/~asnowden/}}

\thanks{SS was supported by a Miller research fellowship. AS was supported by NSF grant DMS-1303082}

\subjclass[2010]{%
13E05, 
13A50.
}

\title[Noetherianity of some degree two tca's]{Noetherianity of some degree two\\ twisted commutative algebras}

\date{September 18, 2015}

\begin{document}

\maketitle

\begin{abstract}
The resolutions of determinantal ideals exhibit a remarkable stability property: for fixed rank but growing dimension, the terms of the resolution stabilize (in an appropriate sense). One may wonder if other sequences of ideals or modules over coordinate rings of matrices exhibit similar behavior. We show that this is indeed the case. In fact, our main theorem is more fundamental in nature: it states that certain large algebraic structures (which are examples of twisted commutative algebras) are noetherian. These are important new examples of large noetherian algebraic structures, and ones that are in some ways quite different from previous examples.
\end{abstract}

\tableofcontents

\section{Introduction}

Let $A_n$ be the coordinate ring of the space of symmetric bilinear forms on $\bC^n$, that is, $\Sym(\Sym^2(\bC^n))$. Inside of $\Spec(A_n)$ is the closed subset $V(I_{n,r})$ of forms of rank at most $r$ defined by the determinantal ideal $I_{n,r}$. The resolution of $A_n/I_{n,r}$ over $A_n$ is explicitly known by a classical result of Lascoux \cite{lascoux} (see also \cite[Chapter 6]{weyman}). The explicit description of the resolution reveals an interesting feature: its terms stabilize as $n$ grows. More precisely, the decomposition of $\Tor_{A_n}^p(A_n/I_{n,r}, \bC)$ into irreducible representations of $\GL_n$ is independent of $n$ for $n \gg p,r$, when one appropriately identifies irreducibles of $\GL_n$ with a subset of those for $\GL_{n+1}$.

Given this observation, one may wonder if the same phenomenon holds true more generally. That is, suppose that for each $n \ge 0$ we have a finitely generated $\GL_n$-equivariant $A_n$-module $M_n$ such that the $M_{\bullet}$ are ``compatible'' in an appropriate sense. Do the resolutions of the $M_n$ stabilize?

The main result of this paper (Theorem~\ref{mainthm}) implies that the answer to this question is ``yes.'' In fact, Theorem~\ref{mainthm} establishes a more fundamental result: compatible sequences of finitely generated equivariant $A_n$-modules are ``noetherian'' in an appropriate sense.

\subsection{Statement of results}

Instead of working with a compatible sequence of $A_n$-modules, we prefer to pass to the limit in $n$ and work with a single module over the ring $\Sym(\Sym^2(\bC^{\infty}))$. This ring, with its $\GL_{\infty}$ action, is an example of a {\bf twisted commutative algebra (tca)}; see \S \ref{sec:defn} for the general definition. Given a tca $A$, there is a notion of (finitely generated) $A$-module, and $A$ is said to be {\bf noetherian} if any submodule of a finitely generated $A$-module is again finitely generated.

Our main result is the following theorem:

\begin{theorem}
\label{mainthm}
The tca's $\Sym(\Sym^2(\bC^{\infty}))$ and $\Sym(\lw^2(\bC^{\infty}))$ are noetherian.
\end{theorem}

We also prove a variant of the above result. A {\bf bivariate tca} is like a tca, but where the group $\GL_{\infty} \times \GL_{\infty}$ acts. We prove:

\begin{theorem}
\label{mainthm2}
The bivariate tca $\Sym(\bC^{\infty} \otimes \bC^{\infty})$ is noetherian.
\end{theorem}

\begin{remark}
\label{remark:FIM}
Let $\mathbf{FIM}$ be the category whose objects are finite sets and where a morphism $X \to Y$ is a pair $(f, \Gamma)$ consisting of an injection $f \colon X \to Y$ and a perfect matching $\Gamma$ on $Y \setminus f(X)$. Then the category of $\Sym(\Sym^2(\bC^{\infty}))$-modules is equivalent to the category of $\mathbf{FIM}$-modules over $\bC$ (see \cite[\S 4.3]{infrank}, where $\mathbf{FIM}$ is called the upwards Brauer category). Thus Theorem~\ref{mainthm} shows that finitely generated $\mathbf{FIM}$-modules are noetherian. This is reminiscent of the noetherianity result for $\mathbf{FI}$-modules (see \cite[Theorem~1.3]{fimodules}), but much more difficult. There are analogous reinterpretations for the other two cases.
\end{remark}

\subsection{Motivation}

We offer a few pieces of motivation for our work.
\begin{itemize}
\item Our main theorems generalize and place into the proper context the stability phenomena observed in the resolutions of determinantal ideals and related ideals (such as those considered in \cite{raicu}).
\item The algebras appearing in Theorems~\ref{mainthm} and~\ref{mainthm2} are closely related to the representation theory of orthogonal and symplectic groups; for example, see \cite{infrank} or Example~\ref{ex:brauer} below. We believe our theorems will have useful applications in this area.
\item The tca's we consider provide important new additions to the growing list of large noetherian algebraic structures; see \S \ref{ss:connection} for further discussion.
\item $\mathbf{FIM}$-modules are formally very similar to the $\FI$-modules studied in \cite{fimodules,fi-noeth}. Numerous examples of $\FI$-modules have been found, and the noetherian property for $\FI$-modules often translates to interesting new theorems about the examples (e.g., representation stability in the cohomology of configuration spaces). We do not currently have analogous examples of $\mathbf{FIM}$-modules, but when examples are found (which we expect), Theorem~\ref{mainthm} will yield interesting new results about them.
\end{itemize}

\begin{example} \label{ex:brauer}
For $\delta \in \bC$ define the {\bf Brauer category} $B(\delta)$ as follows: objects are finite sets, and morphisms are Brauer diagrams, where composition of Brauer diagrams uses the parameter $\delta$. One can regard $\mathbf{FIM}$ as a subcategory of $B(\delta)$, and from this one can deduce noetherianity of $B(\delta)$-modules from Theorem~\ref{mainthm}. Suppose that $\delta=n-m$ with integers $n$ and $m$. Then one obtains an interesting $B(\delta)$-module by $S \mapsto (\bC^{n|m})^{\otimes S}$, where $\bC^{n|m}$ is the super vector space of the indicated super dimension. This module is closely connected to the representation theory of the orthosymplectic Lie algebra $\fosp(n|m)$.  Our theorem shows that any submodule of this module is finitely generated. The second and third author plan to study $B(\delta)$-modules more closely in a future paper, and the noetherian property will be of foundational importance.
\end{example}

\subsection{Connection to previous work}
\label{ss:connection}

Theorem~\ref{mainthm} fits into a theme that has emerged in recent years where large algebraic structures have been found to be noetherian. See \cite{aschenbrennerhillar, cohen, hillarsullivant} for examples of $S_\infty$-equivariant polynomial rings. Some other examples include $\Delta$-modules \cite{delta-mod}, $\FI$-modules \cite{fimodules,fi-noeth} (see also \cite{symc1}), $\FS$-modules \cite{catgb}, $\mathtt{VIC}(R)$-modules \cite{putnamsam}, and certain spaces of infinite matrices \cite{draismakuttler, draismaeggermont, eggermont}.

However, the noetherian results of this paper seem fundamentally more difficult than the previous ones. We do not know how to make this observation precise, but offer the following observation. One can almost always use Gr\"obner bases to reduce a noetherianity problem in algebra to one in combinatorics (see \cite{catgb}). In the previous noetherianity results, the combinatorial problems ultimately concern words in a formal language, and can be easily solved using Higman's lemma. In contrast, the combinatorial problem that naturally arises in the present case (Question~\ref{ques:grobner}) is graph-theoretic, and does not seem approachable by Higman's lemma. Alternatively, this division can be seen in terms of the asymptotics of Hilbert functions: in the previous noetherian results, the Hilbert functions have exponential growth, while in the present case the growth is super-exponential.

Due to this fundamental new difficulty, we have been forced to introduce new methods to prove the main theorem. We believe these will be useful more generally.

\subsection{Outline of proof}
\label{subsection:outline}
We now outline the proof of noetherianity for $A=\Sym(\Sym^2(\bC^{\infty}))$. Let $K=\Frac(A)$ and let $\Mod_A^{\tors}$ be the category of torsion $A$-modules, where we say that an $A$-module $M$ is torsion if $M \otimes_A K=0$. If $I$ is a non-zero ideal of $A$ then the quotient tca $A/I$ is ``essentially bounded,'' and it is not difficult to conclude from this that $A/I$ is noetherian (see Proposition~\ref{prop:eb}); it follows that finitely generated objects of $\Mod_A^{\tors}$ are noetherian.

We next consider the Serre quotient category $\Mod_A/\Mod_A^{\tors}$, which we denote by $\Mod_K$. The intuition for $\Mod_K$ comes from the following picture, which is not rigorous:
\begin{addmargin}{2em}
The scheme $\Spec(A)$ is the space of symmetric bilinear forms on $\bC^{\infty}$. $A$-modules correspond to $\GL_{\infty}$-equivariant quasi-coherent sheaves on $\Spec(A)$. Torsion $A$-modules correspond to sheaves that restrict to zero on the open set $U$ of non-degenerate forms. Thus objects of $\Mod_K$ correspond to equivariant quasi-coherent sheaves on $U$. But such sheaves correspond to representations of $\bO_{\infty}$, since $\GL_{\infty}$ acts transitively on $U$ with stabilizer $\bO_{\infty}$.
\end{addmargin}
The above reasoning is fraught with errors. Nonetheless, it leads to a correct statement: we prove (Theorem~\ref{thm:modK-equiv}) that $\Mod_K$ is equivalent to the category of algebraic representations of $\bO_{\infty}$, as defined in \cite{infrank}. The results of \cite{infrank} can therefore be transferred to $\Mod_K$, and give an essentially complete understanding of this category.

We would now like to piece together what we know about $\Mod_A^{\tors}$ and $\Mod_K$ to deduce the noetherianity of $A$. However, the noetherianity of $A$ is not a formal consequence of what we have so far: we need to use more information about \emph{how} $\Mod_A$ is built out of the two pieces $\Mod_A^{\tors}$ and $\Mod_K$. We proceed in three steps. 
\begin{enumerate}[\indent (1)]
\item We show that if $M$ is a finitely generated torsion $A$-module then $M$ admits a resolution by finitely generated projective $A$-modules (Proposition~\ref{prop:tors-FT}). The essential input here is \cite{raicu}, which explicitly computes the resolutions of certain torsion modules. 
\item We next show that the section functor $\Mod_K \to \Mod_A$, defined as the right adjoint of the localization functor $\Mod_A \to \Mod_K$, takes finite length objects of $\Mod_K$ to finitely generated objects of $\Mod_A$. This follows from step (1) and the structural results for $\Mod_K$ (see Proposition~\ref{prop:S-fin}). 
\item Finally, the noetherianity of $A$ is deduced from (2), and our knowledge of $\Mod_A^{\tors}$ and $\Mod_K$, by a short argument (see Theorem~\ref{thm:A-noeth}).
\end{enumerate}

\begin{remark} \label{rmk:pfctx}
Let us offer some broader context for this proof. Suppose that $X$ is a scheme equipped with an action of a group $G$. We say that $X$ is {\bf topologically $G$-noetherian} if every descending chain of $G$-stable Zariski closed subsets in $X$ stabilizes. We say that $X$ is {\bf (scheme-theoretically) $G$-noetherian} if the analogous statement holds for subschemes\footnote{One should ask that all $G$-equivariant coherent sheaves are noetherian, not just the structure sheaf.}. Suppose that $U$ is a $G$-stable open subscheme of $X$, and let $Z$ be the complement of $U$. One would then like to relate the noetherianity of $X$ to that of $U$ and $Z$.

For topological noetherianity, there is no problem: if $U$ and $Z$ are topologically $G$-noetherian then so is $X$ (see \cite[\S 5]{draismakuttler}). This is a fundamental tool used in various topological noetherianity results, such as \cite{draismaeggermont,draismakuttler,eggermont}. Unfortunately, the analogous statement for scheme-theoretic $G$-noetherianity does not hold: this is why we cannot directly conclude the noetherianity of $A$ from that of $\Mod_A^{\tors}$ and $\Mod_K$.

The main technical innovation in this paper is our method for deducing (in our specific situation) scheme-theoretic noetherianity of $X$ from that of $U$ and $Z$, together with some extra information. This approach is likely to be applicable in other situations, and could be very useful: for instance, if one could upgrade the topological results of \cite{draismakuttler} to scheme-theoretic results, it is likely that one could also get finiteness results for higher syzygies in addition to results about equations (and not just set-theoretic equations).
\end{remark}

\subsection{Twisted graded-commutative algebras}

One can define a notion of twisted graded-commutative algebra, the basic examples being exterior algebras on finite length polynomial representations of $\GL_{\infty}$. The noetherianity problem for these algebras is interesting, and has applications similar to the commutative case. Transpose duality interchanges the algebras $\Sym(\bC^{\infty} \otimes \bC^{\infty})$ and $\lw(\bC^{\infty} \otimes \bC^{\infty})$, and so noetherianity of the latter is an immediate consequence of Theorem~\ref{mainthm2}. However, the noetherianity of $\lw(\Sym^2(\bC^{\infty}))$ and $\lw(\lw^2(\bC^{\infty}))$ cannot be formally deduced from the results of this paper. We treat these algebras in a follow-up paper \cite{NSS}. The main ideas are the same, but the details are more complicated: for example, while $\Sym(\Sym^2(\bC^{\infty}))$ is closely related to the orthogonal group $\bO_{\infty}$, the algebra $\lw(\Sym^2(\bC^{\infty}))$ is closely related to the periplectic superalgebra $\mathfrak{pe}_{\infty}$.

\subsection{Open questions}

We list a number of open problems related to this paper.
\begin{enumerate}[\indent (1)]
\item Theorem~\ref{mainthm} states that the tca $\Sym(V)$ is noetherian when $V$ is an irreducible polynomial representation of degree~2. It would be natural to generalize this result by allowing $V$ to be a finite length representation of degree $\le 2$. Eggermont \cite{eggermont} has shown that these tca's are topologically noetherian (i.e., radical ideals satisfy the ascending chain condition). This suggests that they are all noetherian. However, new ideas are needed to actually prove this.
\item It is desirable to have results (either positive or negative) when $V$ has degree $>2$. One might begin by trying to prove topological noetherianity for degree~3 representations. The third author is currently investigating this with H.~Derksen and R.~Eggermont.
\item Are the characteristic $p$ analogs of the tca's considered in this paper noetherian? Our methods do not apply there. We point out that there are two versions of tca's in positive characteristic: one defined in terms of polynomial representations, and one defined in terms of symmetric groups.
\item Theorem~\ref{mainthm} shows that $A=\Sym(\Sym^2(\bC^{\infty}))$ is noetherian if we make use of the $\GL_{\infty}$ action. On the other hand, it is known that $A$ is \emph{not} noetherian if one only makes use of the $S_{\infty}$ action \cite[Example~2.4]{draisma-notes}. What happens for other groups? Is $A$ noetherian with respect to $\bO_{\infty}$ or $\Sp_{\infty}$?
\item In \S \ref{ss:ft}, we show that torsion modules over $\Sym(\bC^{\infty} \otimes \bC^{\infty})$ satisfy the property (FT) by appealing to \cite{raicu}, which explicitly computes the resolutions of certain torsion modules. We also show that torsion modules over $\Sym(\Sym^2(\bC^{\infty}))$ satisfy (FT), but deduce this by a rather clumsy argument from the previous case since the analog of \cite{raicu} is not known in this case. We therefore believe that carrying out the analog of \cite{raicu} for $\Sym(\Sym^2(\bC^{\infty}))$ would be a worthwhile undertaking.
\item Question~\ref{ques:grobner} is an interesting and purely combinatorial question that is needed for the Gr\"obner approach to Theorem~\ref{mainthm}.
\end{enumerate}

\subsection{Outline of paper}

In \S\ref{sec:prelim}, we review definitions and prove some general properties of tca's. These include generalities on the localization functor $\Mod_A \to \Mod_K$ and the section functor $S \colon \Mod_K \rightarrow \Mod_A$ used in the proof of the main result. In \S\ref{sec:modK-groups} we prove that for the specific algebras under consideration, the Serre quotient category $\Mod_K$ can be described in terms of representations of infinite rank classical groups. The proofs of Theorem~\ref{mainthm} and Theorem~\ref{mainthm2} are in \S\ref{sec:proof}. Finally, \S\ref{sec:grobner} discusses an incomplete Gr\"{o}bner theoretic approach to the main theorems.

\begin{remark} \label{rmk:transp}
Transpose duality \cite[\S 7.4]{expos} interchanges the two algebras in Theorem~\ref{mainthm}, so it suffices to prove the noetherianity for either one. We give arguments for both when convenient, but sometimes omit details for $\Sym(\lw^2(\bC^{\infty}))$.
\end{remark}

\section{Generalities on tca's} \label{sec:prelim}

\subsection{Definitions} \label{sec:defn}

A representation of $\GL_{\infty}$ is {\bf polynomial} if it appears as a subquotient of a (possibly infinite) direct sum of representations of the form $(\bC^{\infty})^{\otimes k}$. Polynomial representations are semi-simple, and the simple ones are the $\bS_{\lambda}(\bC^{\infty})$, where $\bS_{\lambda}$ denotes the Schur functor corresponding to the partition $\lambda$. A polynomial representation is said to have {\bf finite length} if it is a direct sum of finitely many simple representations. See \cite{expos} for details.

A {\bf twisted commutative algebra} (tca) is a commutative associative unital $\bC$-algebra $A$ equipped with an action of $\GL_{\infty}$ by $\bC$-algebra homomorphisms such that $A$ forms a polynomial representation of $\GL_{\infty}$. Alternatively, $A$ is a polynomial functor from vector spaces to commutative algebras \cite[Theorem 5.4.1]{expos}; when used in this perspective, we use $A(V)$ to denote its value on a vector space $V$.

We write $\vert A \vert$ when we want to think of $A$ simply as a $\bC$-algebra, and forget the $\GL_{\infty}$ action. An {\bf $A$-module} is an $\vert A \vert$-module $M$ equipped with an action of $\GL_{\infty}$ that is compatible with the one on $A$ (i.e., $g(a x) = (g a)(g x)$ for $g \in \GL_{\infty}$, $a \in A$, and $x \in M$) and such that $M$ forms a polynomial representation of $\GL_{\infty}$. An {\bf ideal} of $A$ is an $A$-submodule of $A$, i.e., a $\GL_{\infty}$-stable ideal of $\vert A \vert$. We denote the category of $A$-modules by $\Mod_A$. We write $\vert M \vert$ when we want to think of $M$ as a module over $\vert A \vert$, forgetting its $\GL_\infty$-structure.

We say that $A$ is {\bf finitely generated} if $\vert A \vert$ is generated as a $\bC$-algebra by the $\GL_{\infty}$ orbits of finitely many elements. Equivalently, $A$ is finitely generated if it is a quotient of a tca of the form $\Sym(V)$, where $V$ is a finite length polynomial representation of $\GL_{\infty}$. An $A$-module $M$ is {\bf finitely generated} if it is generated as an $\vert A \vert$-module by the $\GL_{\infty}$ orbits of finitely many elements. Equivalently, $M$ is finitely generated if it is a quotient of an $A$-module of the form $A \otimes V$, where $V$ is a finite length polynomial representation of $\GL_{\infty}$. We note that the $A \otimes V$ are exactly the projective $A$-modules. An $A$-module if {\bf noetherian} if every submodule is finitely generated. We say that $A$ is {\bf noetherian} (as an algebra) if every finitely generated $A$-module is noetherian.

\begin{remark} \label{rmk:weak-noeth}
We say that $A$ is {\bf weakly noetherian} if it is noetherian as a module over itself, i.e., if ideals of $A$ satisfy ACC. Of course, noetherian implies weakly noetherian. However, it is not clear if weakly noetherian implies noetherian: not every $A$-module is a quotient of a direct sum of copies of $A$, due to the equivariance, and so there is no apparent way to connect the noetherianity of $A$ as an $A$-module to that of general modules.
\end{remark}

There are ``bivariate'' versions of the above concepts. A representation of $\GL_{\infty} \times \GL_{\infty}$ is {\bf polynomial} if it appears as a subquotient of a (possibly infinite) direct sum of representations of the form $(\bC^{\infty})^{\otimes a} \otimes (\bC^{\infty})^{\otimes b}$. Polynomial representations are again semi-simple, and the simple ones are the $\bS_{\lambda}(\bC^{\infty}) \otimes \bS_{\mu}(\bC^{\infty})$. A {\bf bivariate tca} is a commutative associative unital $\bC$-algebra $A$ equipped with an action of $\GL_{\infty} \times \GL_{\infty}$ by $\bC$-algebra homomorphisms such that $A$ forms a polynomial representation of $\GL_{\infty} \times \GL_{\infty}$. The remaining definitions in the bivariate case should now be clear.

Since $\GL_{\infty}$ sits inside of $\GL_{\infty} \times \GL_{\infty}$ (diagonally), any action of $\GL_{\infty} \times \GL_{\infty}$ can be restricted to one of $\GL_{\infty}$. Thus bivariate tca's can be regarded as tca's, and similarly for modules. This restriction process preserves finite generation (of algebras and modules) since the tensor product of finite length polynomial representations is again finite length.

\subsection{Annihilators} \label{ss:ann}

Let $A$ be a tca and $M$ be an $A$-module. The {\bf annihilator} of $M$, denoted $\ann(M)$, is the set of elements $a \in A$ such that $am=0$ for all $m \in M$. This is an ideal of $\vert A \vert$ and $\GL_{\infty}$ stable, and thus an ideal of $A$.

\begin{proposition}
\label{prop:ann}
Let $A$ be a tca and let $M$ be an $A$-module. Suppose $am=0$ for some $a \in A$ and $m \in M$. Then there exists an integer $n$, depending only on $m$, such that $a^n (gm)=0$ for all $g \in \GL_{\infty}(\bC)$.
\end{proposition}

\begin{proof}
First, we claim that $a^{k+1} X_k \cdots X_1 m=0$ for any $X_1, \ldots, X_k \in \gl_{\infty}$. We proceed by induction on $k$. The $k=0$ case is simply the statement $am=0$, which is given. Suppose now that $a^k X_{k-1} \cdots X_1 m=0$. Applying $X_k$, we obtain
\begin{displaymath}
k a^{k-1} (X_k a) (X_{k-1} \cdots X_1 m) + a^k (X_k \cdots X_1 m) = 0.
\end{displaymath}
Multiplying by $a$ kills the first term and shows $a^{k+1} X_k \cdots X_1 m=0$. This completes the proof of the claim.

Let $M' \subset M$ be the $\GL_{\infty}$ representation generated by $m$. Suppose that $a$ belongs to $A(V)$ with $V \subset \bC^\infty$. Pick $m' \in M'$ that also generates $M'$ and belongs to $M'(U)$ with $U \cap V = 0$. We can write $m' = Xm$ for some $X \in \cU(\gl_\infty)$, and so the claim shows that $a^n m' = 0$ for some $n$. We claim that this $n$ works for all elements in $M'$. Indeed, given any $g \in \GL_\infty$, we can find $f \colon \bC^\infty \to \bC^\infty$ such that $f$ agrees with $g$ on $U$ and is the identity on $V$. We then have $f_*(a) = a$ and $f_*(m') = gm'$, and so $0 = f_*(a^n m') = a^n (gm')$.
\end{proof}

\begin{corollary} \label{cor:nz-ann}
Suppose $\vert A \vert$ is a domain. Let $M$ be a finitely generated $A$-module such that $M \otimes_A \Frac(A) = 0$. Then $\ann{M} \ne 0$.
\end{corollary}

\begin{proof}
Let $m_1, \ldots, m_r$ be generators for $M$. Since $M \otimes_A \Frac(A)=0$, we can find $a \ne 0$ in $A$ such that $am_i=0$ for all $1 \le i \le r$. By the proposition, there exists $n>0$ such that $a^n (gm_i)=0$ for all $1 \le i \le r$ and all $g \in \GL_{\infty}$. Thus $0 \ne a^n \in \ann(M)$.
\end{proof}

\subsection{Essentially bounded tca's}

We say that a polynomial representation $V$ of $\GL_{\infty}$ is {\bf essentially bounded} if there exist integers $r$ and $s$ such that for any simple $\bS_\lambda(\bC^{\infty})$ appearing in $V$ we have $\lambda_r \le s$. Similarly, we say that a polynomial representation $V$ of $\GL_{\infty} \times \GL_{\infty}$ is {\bf essentially bounded} if there exist integers $r$ and $s$ such that for any simple $\bS_{\lambda}(\bC^{\infty}) \otimes \bS_{\mu}(\bC^{\infty})$ appearing in $V$ we have $\lambda_r \le s$ and $\mu_r \le s$. The Littlewood--Richardson rule \cite[(2.14)]{expos} implies that the tensor product of essentially bounded representations is again essentially bounded. In particular, if $V$ is an essentially bounded representation of $\GL_{\infty} \times \GL_{\infty}$ then its restriction to the diagonal $\GL_{\infty}$ is still essentially bounded. Note also that any finite length representation is essentially bounded.

\begin{proposition}
\label{prop:eb}
Let $A$ be a finitely generated and essentially bounded (bivariate) tca. Then $A$ is noetherian.
\end{proposition}

\begin{proof}
We treat only the univariate case, the bivariate case is similar. Let $P$ be a finitely generated projective $A$-module. Note that $P$ is essentially bounded. We must show that $P$ is noetherian. Suppose that every partition appearing in $P$ has at most $r$ rows and at most $s$ columns. Let $\bC^{r|s}$ be a super vector space with $r$-dimensional even part and $s$-dimensional odd part. 

For any symmetric monoidal category $\cC$ and choice of object $V \in \cC$, there is a symmetric monoidal functor $\Rep^\pol(\GL_\infty) \to \cC$ that sends $\bS_\lambda(\bC^\infty)$ to $\bS_\lambda(V)$. We apply this with $\cC$ the category of super vector spaces, equipped with the usual tensor product and the signed symmetry (see \cite[(7.3.3)]{expos}), and $V = \bC^{r|s}$. We thus obtain a natural map
\begin{displaymath}
\{ \text{$A$-submodules of $P$} \} \to \{ \text{$A(\bC^{r \mid s})$-submodules of $P(\bC^{r \mid s})$} \}.
\end{displaymath}
It follows from \cite[Theorem 3.20]{BR} that this map is injective. Since $A(\bC^{r \mid s})$ is a finitely generated superalgebra, the finitely generated module $P(\bC^{r \mid s})$ is noetherian. Thus the right side satisfies ACC and so the left side does as well.
\end{proof}

\begin{remark}
This argument is modeled on the discussion in \cite[\S 9.1]{expos}.
\end{remark}

\subsection{Serre quotients} \label{sec:serrequotient}

Let $A$ be a tca with $\vert A \vert$ a domain, and let $K=\Frac(\vert A \vert)$. The field $K$ has an action of $\GL_{\infty}$, and we write $\vert K \vert$ when we want to disregard this action. A {\bf $K$-module} is a $\vert K \vert$-vector space $V$ equipped with a compatible action of $\GL_{\infty}$ such that $V$ is spanned over $\vert K \vert$ by polynomial elements (i.e., elements generating a polynomial $\bC$-subrepresentation). We write $\Mod_K$ for the category of $K$-modules. If $V$ is a polynomial representation of $\GL_{\infty}$ then $K \otimes V$ is a $K$-module. All $K$-modules are quotients of such $K$-modules. Note, however, that $K \otimes V$ is usually not projective as a $K$-module.

An $A$-module $M$ is {\bf torsion} if $M \otimes_A K=0$. Write $\Mod_A^{\rm tors}$ for the category of torsion modules. We let $\Mod_A^{\gen}$ be the Serre quotient $\Mod_A / \Mod_A^{\rm tors}$, and we let $T \colon \Mod_A \to \Mod_A^{\gen}$ be the localization functor. The functor $\Mod_A \to \Mod_K$ given by $M \mapsto M \otimes_A K$ is exact and kills $\Mod_A^{\rm tors}$, and thus induces an exact functor $F \colon \Mod_A^{\gen} \to \Mod_K$. Note that $F(T(M))=M \otimes_A K$, by definition. Given a $K$-module $M$, let $S(M)=M^{\pol}$ be the set of polynomial elements in $M$. This is naturally an $A$-module, and the resulting functor $S \colon \Mod_K \to \Mod_A$ is right adjoint to $FT$. The following diagram summarizes the situation.
\begin{displaymath}
\xymatrix{
& \Mod_A \ar[ld]_T \\
\Mod_A^{\gen} \ar[rr]^F && \Mod_K \ar[lu]_S }
\end{displaymath}

\begin{proposition} \label{prop:FTS-id}
Let $M$ be a $K$-module. Then the natural map $M^{\pol} \otimes_A K \to M$ is an isomorphism of $K$-modules. In particular, we have a natural isomorphism $FTS=\id$.
\end{proposition}

\begin{proof}
Injectivity is free. For surjectivity, pick $v \in M$. By definition, we can write $v = \sum_{i=1}^n x_i w_i$, where $x_i \in K$ and $w_i \in M$ is polynomial. Again, by definition, we can write $x_i=a_i/b_i$, where $a_i$ and $b_i$ are polynomial elements of $K$. We therefore have $v=b^{-1} w$, where $b = \prod_{i=1}^n b_i \in K$ and $w = \sum_{i=1}^n (a_ib/b_i) w_i$ is a polynomial element of $M$.
\end{proof}

\begin{proposition} \label{prop:modKequiv}
The functor $F$ is an equivalence of categories.
\end{proposition}

\begin{proof}
Proposition~\ref{prop:FTS-id} shows that $F$ has a right quasi-inverse, and so is therefore essentially surjective and full. We show that $F$ is faithful. Let $M$ and $N$ be $A$-modules, and consider a morphism $\wt{f} \colon M \to N$ in $\Mod_A^{\gen}$ mapping to $0$ in $\Mod_K$. Write $\wt{f}=T(f)$ for some morphism $f \colon M' \to N/N'$ in $\Mod_A$, where $M'$ and $N'$ are submodules of $M$ and $N$ with $M/M'$ and $N'$ torsion. Since $f \colon M' \otimes K \to N/N' \otimes K$ is $0$, it follows that the image of $f$ is a torsion submodule of $N/N'$, and therefore of the form $N''/N'$, where $N''$ is a torsion-submodule of $N$ containing $N'$. But then the image $f'$ of $f$ in $\Hom(M', N/N'')$ is $0$, and since $\wt{f}=T(f)=T(f')$ we have $\wt{f}=0$.
\end{proof}

\begin{proposition} \label{prop:poly-sat}
Let $W$ be a finite length polynomial representation with $W^{\GL_\infty} = 0$. Set $A =\Sym(W)$. Then $S(K \otimes V)=A \otimes V$ for any polynomial representation $V$.
\end{proposition}

\begin{proof}
Suppose that $x = \sum_{i=1}^s (f_i/g) \otimes v_i$ is a polynomial element of $K \otimes V$, written in lowest terms (that is, $\gcd(g, f_1, \ldots, f_s) = 1$ and $\{v_1, \ldots, v_s\}$ is linearly independent). Let $m \gg 0$ be such that $g$ and each $f_i$ belong to $A(\bC^m)$, and let $n=m+1$.  We can think of $x$ as a section of a vector bundle on $\bC^n$ having a pole along the divisor $g=0$. Since $x \in (K\otimes V)^{\pol}$, it generates a finite dimensional representation of $\GL_n$. Let $\sum_{k}(f_{j,k}/g_{j}) \otimes v_{j,k}$ for $1 \leq j \leq r$ be a basis; then every element can be written with common denominator $g_1\cdots g_r$. In particular, the $\GL_n$-orbit of the divisor $g=0$ is contained in $g_1 \cdots g_r =0$ and hence is finite. But $\GL_n$ is connected, so the irreducible components of $g=0$ are preserved. Thus $g$ is semi-invariant under $\GL_n$. Any one-dimensional polynomial representation of $\GL_n$ must be of the form $\bS_{d,...,d}(\bC^n)$. But $g \in A(\bC^m)$ and is nonzero, and so it must be the case that $g$ is actually invariant under $\GL_n$ ($d$ must be zero because otherwise $\bS_{d,...,d}(\bC^n) = 0$), and thus under $\GL_{\infty}$. Since $A^{\GL_{\infty}}=\bC$, we conclude that $g$ is constant, and so $x \in A \otimes V$, as required.
\end{proof}

There is also a version of the above discussion for bivariate tca's. The statements and proofs are nearly identical.

\section{$\Mod_K$ and algebraic representations} \label{sec:modK-groups}

\subsection{The main theorem and its consequences}

A representation of $\bO_{\infty}=\bigcup_{n \ge 1} \bO_n$ is {\bf algebraic} if it appears as a subquotient of a (possibly infinite) direct sum of tensor powers of the standard representation $\bC^{\infty}$. We write $\Rep(\bO_{\infty})$ for the category of such representations. This category was studied in \cite{infrank}.

We let $A=\Sym(\Sym^2(\bC^{\infty}))$ and $K=\Frac(A)$ until \S \ref{ss:other}. We let $e_1, e_2, \ldots$ be a basis for $\bC^{\infty}$, and let $x_{i,j}=e_i e_j$, so that $A=\bC[x_{i,j}]$. Define $\fm \subset \vert A \vert$ to be the ideal generated by $x_{i,i}-1$ and $x_{i,j}$ for $i \ne j$. This ideal is not stable by $\GL_{\infty}$, but is stable by $\bO_{\infty}$. The quotient $A/\fm$ is isomorphic to $\bC$. For an $A$-module $M$, define $\wt{\Phi}(M)=M/\fm M$. This is naturally a representation of $\bO_{\infty}$. The main result of \S \ref{sec:modK-groups} is the following theorem (see \S \ref{ss:other} for analogous results in the other two cases):

\begin{theorem} \label{thm:modK-equiv}
The functor $\wt{\Phi}$ induces an equivalence of categories $\Phi \colon \Mod_K \to \Rep(\bO_{\infty})$.
\end{theorem}

We give the proof in the following subsections. The precise definition of $\Phi$ is given in \S \ref{ss:phi-defn}. For now, we note the following consequences of this theorem:

\begin{corollary} \label{cor:list}
We have the following:
\begin{enumerate}[\indent \rm (a)]
\item Finitely generated objects of $\Mod_K$ have finite length.
\item If $V$ is a finite length polynomial representation of $\GL_{\infty}$ then $K \otimes V$ is a finite length injective object of $\Mod_K$, and all finite length injective objects have this form.
\item Associating to $\lambda$ the socle of $K \otimes \bS_{\lambda}(\bC^{\infty})$ gives a bijection between partitions and isomorphism classes of simple objects of $\Mod_K$. 
\item Every finite length object $M$ of $\Mod_K$ has a finite injective resolution $M \to I_{\bullet}$ where each $I_k$ is a finite length injective object.
\end{enumerate}
\end{corollary}

\begin{proof}
These properties are proven for $\Rep(\bO_{\infty})$ in \cite{infrank}:

(a) \cite[Proposition 4.1.5]{infrank}, 

(b,c) \cite[Proposition 4.2.9]{infrank}, 

(d) dualize the explicit projective resolutions in \cite[(4.3.9)]{infrank}. \qedhere
\end{proof}

\subsection{Local structure at $\fm$ of $A$-modules}
\label{ss:struc-at-m}

The main result of this section is the following:

\begin{proposition} \label{prop:local-at-m}
Let $M$ be an $A$-module. Then $M_{\fm}$ is a free $A_{\fm}$-module.
\end{proposition}

Let $\rM_{\infty}$ be the set of infinite complex matrices, indexed by $\bZ_{\ge 0}$. Let $\ol{B} \subset \rM_{\infty}$ be the set of upper triangular matrices, and let $B \subset \ol{B}$ be the group of invertible upper triangular matrices. Let $b_{i,j} \colon \rM_{\infty} \to \bC$ be the function taking the $(i,j)$ matrix entry. We let $\bC[\ol{B}]$ be the polynomial ring $\bC[b_{i,j}]_{i \le j}$, and we let $\bC[B]=\bC[\ol{B}][b_{i,i}^{-1}]$. Elements of $V \otimes \bC[B]$ can be thought of as (certain) functions $B \to V$.

If $V$ is a polynomial representation of $\GL_{\infty}$ then every $v \in V$ spans a finite dimensional subrepresentation of $B$. It follows that we can give $V$ the structure of a $\bC[B]$-comodule, that is, we have a map $V \to V \otimes \bC[B]$. Explicitly, this map takes $v$ to the function $B \to V$ given by $b \mapsto bv$. In fact, the image of the comultiplication map is contained in $V \otimes \bC[\ol{B}]$.

Let $H=B \cap \bO_{\infty}$. Explicitly, $H$ is the group of diagonal matrices with diagonal entries $\pm 1$, almost all of which are 1. If $V$ is a polynomial representation of $\GL_{\infty}$ then the map $V \to V \otimes \bC[\ol{B}]$ above actually lands in the $H$-invariants of the target. Here we let $B$, and $H$, act on $\bC[\ol{B}]$ by right translation.

Let $M$ be an $A$-module. We then obtain a map
\begin{displaymath}
M \to M \otimes \bC[\ol{B}] \to M/\fm M \otimes \bC[\ol{B}].
\end{displaymath}
The image lands in the $H$-invariants (note that $\fm$ is $H$-stable, so $H$ still acts on $M/\fm M$), and so we have a map
\begin{displaymath}
\varphi_M \colon M \to (M/\fm M \otimes \bC[\ol{B}])^H.
\end{displaymath}
We now study this map. We first treat the case where $M=A$. Then $A/\fm A=\bC$, and so our map takes the form
\begin{displaymath}
\varphi_A \colon A \to \bC[\ol{B}]^H
\end{displaymath}
The invariant ring $\bC[\ol{B}]^H$ is easily seen to be the subring of $\bC[\ol{B}]$ generated by the $b_{i,j} b_{i,k}$, with $i \le j,k$. Since $\ol{B}$ acts on $A$ by algebra homomorphisms, the map $A \to A \otimes \bC[\ol{B}]$ is an algebra homomorphism, and so $\varphi_A$ is an algebra homomorphism as well. Due to this, it suffices to understand where the generators $x_{i,j}$ go. For $m \in \ol{B}$, we have
\begin{displaymath}
me_i = \sum_{k \le i} b_{k,i}(m) e_k,
\end{displaymath}
and so
\begin{displaymath}
m(e_i e_j) = \left( \sum_{k \le i} b_{k,i}(m) e_k \right) \left( \sum_{\ell \le j} b_{\ell,j}(m) e_{\ell} \right)
=\sum_{k \le i, \ell \le j} b_{k,i}(m) b_{\ell, j}(m) e_k e_{\ell}.
\end{displaymath}
Thus the map $A \to A \otimes \bC[\ol{B}]$ is given by
\begin{displaymath}
x_{i,j} \mapsto \sum_{k \le i,\ \ell \le j} b_{k,i} b_{\ell,j} x_{k,\ell}.
\end{displaymath}
To compute $\varphi_A$, we now apply the homomorphism $A \to A/\fm A=\bC$, which takes $x_{i,j}$ to $\delta_{i,j}$. Set $X_{i,j}=\varphi(x_{i,j})$, we thus find
\begin{displaymath}
X_{i,j} = \varphi(x_{i,j}) = \sum_{k \le i,j} b_{k,i} b_{k,j}.
\end{displaymath}

\begin{proposition}
\label{prop:phiA}
The localization of $\varphi_A$ at $\fm$ is an isomorphism.
\end{proposition}

\begin{proof}
Let $\fm^e$ be the extension of $\fm$ to $\bC[\ol{B}]^H$ via $\varphi_A$. Let $i \le j$. We have $X_{i,j}=b_{i,i} b_{i,j}+X_{i,j}'$, where $X_{i,j}'$ is the sum of the $b_{k,i} b_{k,j}$ with $k<i$. Since $X_{i,j} \in \fm$ for $i \ne j$ and $X_{i,i}-1 \in \fm$, an easy inductive argument shows that $b_{i,i}^2-1 \in \fm$ and $b_{i,j}b_{i,k} \in \fm$ if $i \ne j$ or $i \ne k$. In particular, $b_{i,i}^2$ is a unit in the localization. The expression $X_{i,j} X_{i,k} = b_{i,i}^2 b_{i,j} b_{i,k} + \cdots$ (where the missing terms involve only smaller variables) shows, inductively, that $b_{i,j} b_{i,k}$ belongs to the image of $\varphi_A$ localized at $\fm$. Since these generate $\bC[\ol{B}]^H$, the result follows. (It is easy to see that $\varphi_A$, and hence its localization, is injective.)
\end{proof}

A {\bf monomial character} of $H$ is a homomorphism $H \to \bC^{\times}$ of the form $(z_1, z_2, \ldots) \mapsto z_1^{n_1} z_2^{n_2} \cdots$ where the $n_i$ are integers (it suffices to consider $n_i \in \{0,1\}$) and $n_i=0$ for $i \gg 0$. A representation of $H$ is {\bf admissible} if it is a sum of monomial characters.

\begin{proposition} \label{prop:H-ad}
Let $V$ be an admissible representation of $H$. The localization of $(V \otimes \bC[\ol{B}])^H$ at $\fm$ is a free $A_{\fm}$-module, and the fiber at $\fm$ is canonically isomorphic to $V$.
\end{proposition}

\begin{proof}
It suffices to treat the case where $V$ is one dimensional. Let $\chi = z_{i_1}\cdots z_{i_r}$ be the corresponding (monomial) character. Then an argument similar to the one in the proof of Proposition~\ref{prop:phiA} shows that for any nonzero $v \in V$, the element $v \otimes (b_{i_1, i_1} \cdots b_{i_r, i_r})$ is an $H$ invariant and the localization of $(V \otimes \bC[\ol{B}])^H$ at $\fm$ is a free $A_{\fm}$-module generated by $v \otimes (b_{i_1, i_1} \cdots b_{i_r, i_r})$. The second statement follows immediately from this.
\end{proof}

Now let $M$ be an $A$-module, and consider the map
\begin{displaymath}
\varphi_M \colon M \to (M/\fm M \otimes \bC[\ol{B}])^H.
\end{displaymath}
The target is naturally a module over the ring $\bC[\ol{B}]^H$, which is itself an $A$-algebra, and one easily verifies that $\varphi_M$ is a map of $A$-modules.

\begin{proposition} \label{prop:varphiM}
Let $M$ be an $A$-module. The localization of $\varphi_M$ at $\fm$ is an isomorphism.
\end{proposition}

\begin{proof}
Note that for any $M$, the quotient $M/\fm M$ is an admissible representation of $H$. Since such representations are semi-simple, it follows that the target of $\phi_M$ commutes with direct limits in $M$. It therefore suffices to treat the case where $M$ is finitely generated as an $A$-module. Let $N=(M/\fm M \otimes \bC[\ol{B}])^H_{\fm}$, and let $R$ be the kernel of $(\varphi_M)_{\fm}$. Since $M/\fm M$ is an admissible representation of $H$, Proposition~\ref{prop:H-ad} shows that $N$ is a free $A_{\fm}$-module whose fiber at $\fm$ is isomorphic to $M/\fm M$. It follows that $(\varphi_M)_{\fm}$ is a surjection, since it is a surjection mod $\fm$ and $N$ is free. We thus have an isomorphism $M_{\fm}=R \oplus N$, which shows that $R$ is finitely generated. Since $(\varphi_M)_{\fm}$ induces an isomorphism on the fiber at $\fm$, we see that $R/\fm R=0$. Thus $R=0$ by Nakayama's lemma, which completes the proof.
\end{proof}

Proposition~\ref{prop:local-at-m} follows from the above proposition, since as noted in the above proof, the target of $(\varphi_M)_{\fm}$ is a free $A_{\fm}$-module.

\subsection{Definition of $\Phi$}
\label{ss:phi-defn}

We begin with some simple observations.

\begin{lemma}
If $V$ is a polynomial representation of $\GL_{\infty}$ then $\wt{\Phi}(A \otimes V)$ is isomorphic to the restriction of $V$ to $\bO_{\infty}$. For any $A$-module $M$, $\wt{\Phi}(M)$ is an algebraic representation of $\bO_{\infty}$.
\end{lemma}

\begin{proof}
The first part is clear. For the second part, pick a surjection $A \otimes V \to M$ of $A$-modules. Since $\wt{\Phi}$ is right exact, there is an induced surjection $V \to\wt{\Phi}(M)$. As any quotient of an algebraic representation is algebraic, the result follows.
\end{proof}

\begin{lemma} \label{lem:m}
If $I$ is a non-zero ideal of $A$ then $I+\fm=A$.
\end{lemma}

\begin{proof}
Suppose $I$ is a non-zero ideal of $A$. Let $A_n=\Sym(\Sym^2(\bC^n))$, regarded as a subring of $\vert A \vert$. Then $A$ is the union of the $A_n$, and so for $n$ sufficiently large, $I'=I \cap A_n$ is a non-zero $\GL_n$-stable ideal of $A_n$. Of course, $\fm'=\fm \cap A_n$ is a maximal ideal of $A_n$. The scheme $\Spec(A_n)$ is the space of symmetric bilinear forms on $\bC^n$, and $\fm' \in \Spec(A_n)$ represents the sum of squares form, which has maximal rank. Since $V(I')$ is a proper closed $\GL_n$-stable subset of $\Spec(A_n)$, it cannot contain any form of maximal rank (as the orbit of any such form is dense), and so $I' \not\subset \fm'$. It follows that $I \not\subset \fm$, and so $I+\fm=A$.
\end{proof}

\begin{lemma}
If $M$ is a torsion $A$-module then $\wt{\Phi}(M)=0$.
\end{lemma}

\begin{proof}
Since $\wt{\Phi}$ commutes with direct limits, it suffices to treat the case where $M$ is finitely generated and torsion. By Corollary~\ref{cor:nz-ann}, $M$ has non-zero annihilator $I$, and $I+\fm=A$ by the Lemma~\ref{lem:m}. Thus
\begin{displaymath}
M/\fm M = M \otimes_{A/I} (A/(I+\fm)) = 0. \qedhere
\end{displaymath}
\end{proof}

\begin{lemma} \label{prop:wtPhi-exact}
The functor $\wt{\Phi}$ is exact.
\end{lemma}

\begin{proof}
This follows immediately from Proposition~\ref{prop:local-at-m}.
\end{proof}

Thus $\wt{\Phi}$ is an exact functor killing $\Mod_A^{\tors}$. It follows that $\wt{\Phi}$ factors through the Serre quotient $\Mod_A/\Mod_A^{\tors}$, which we identify with $\Mod_K$. In other words, there exists an exact functor $\Phi \colon \Mod_K \to \Rep(\bO_{\infty})$, unique up to isomorphism, such that $\wt{\Phi}(M)=\Phi(M \otimes_A K)$. Since $\wt{\Phi}$ is compatible with direct limits, so is $\Phi$.

\subsection{Proof of Theorem~\ref{thm:modK-equiv}}

We now prove that $\Phi$ is an equivalence. We first prove that it is faithful, then full, and finally essentially surjective.

\begin{lemma} \label{lem:faithful}
$\Phi$ is faithful.
\end{lemma}

\begin{proof}
Let $f \colon M \to N$ be a map of $A$-modules, and suppose $\wt{\Phi}(f)=0$. The square
\begin{displaymath}
\xymatrix{
M \ar[r]^-{\varphi_M} \ar[d]_f & (M/\fm M \otimes \bC[\ol{B}])^H \ar[d]^{\ol{f} \otimes 1} \\
N \ar[r]^-{\varphi_N} & (N/\fm N \otimes \bC[\ol{B}])^H}
\end{displaymath}
commutes. Since $\wt{\Phi}(f)=\ol{f}=0$, the right map is 0. Since $\varphi_M$ and $\varphi_N$ are isomorphisms after localizing at $\fm$, the induced map $f \colon M_{\fm} \to N_{\fm}$ is 0. This implies that the induced map $f \colon M \otimes_A K \to N \otimes_A K$ is 0, and so $f=0$ in $\Mod_K$. This shows that $\Phi$ is faithful.
\end{proof}

In what follows, we give $\GL_{\infty}$ and $B$ the direct limit topology (thinking of them as the direct limits of $\GL_n$ and $B \cap \GL_n$ in the Zariski topology).

\begin{lemma} \label{lem:V}
Let $M$ be an $A$-module and let $x \in M \otimes_A K$. Then there exists a dense Zariski open subset $V$ of $B$ such that $bx \in M_{\fm}$ for all $b \in V$.
\end{lemma}

\begin{proof}
Given $h \in \bC[B]$, let $U_h = \{ b \in B \mid h(b) \ne 0\}$ be the corresponding Zariski open subset of $B$. Let $V = \{b \in B \mid b x  \in M_{\fm}\}$. We can find nonzero $a \in A$ such that $ax \in M$. Note that $b \in U_{\phi_A(a)}$ if and only if $b a \notin \fm$; since $\phi_A(a)$ is not the zero function by Proposition~\ref{prop:phiA}, we can find such a $b$. Then $bx \in M_\fm$ and so $V \ne \emptyset$.

We claim that $V$ is open. Suppose $b \in V$ and write $b x = \sum_i m_i \otimes (f_i/c)$ with $m_i \in M$, $f_i \in A$, and $c \in A \setminus \fm$. Then $1 \in U_{\phi_A(c)}$ and $b' b x \in M_\fm$ for each $b' \in U_{\phi_A(c)}$. So $U_{\phi_A(c)} b \subseteq V$, showing that $V$ is open. 
Finally, since $B$ is a directed union of irreducible spaces, a nonempty open subset, like $V$, is dense.
\end{proof}

\begin{lemma} \label{lem:open}
Let $M$ be an $A$-module, and let $x \in M \otimes_A K$. Suppose that there exists a dense Zariski open subset $U$ of $B$ such that for all $b \in U$ we have $bx \in M_{\fm}$ and $\ol{bx}=0$, where the overline denotes reduction mod $\fm$. Then $x=0$.
\end{lemma}

\begin{proof}
Replacing $x$ with $ax$, for an appropriate $a \in A$, it suffices to treat the case $x \in M$. Then $b \mapsto \ol{bx}$ defines a function $B \to M/\fm M$ which is continuous for the Zariski topology. The hypothesis implies that it vanishes on a dense subset of $B$, and therefore it vanishes on all of $B$. So $\varphi_M(x)=0$, and so $x=0$ since $\varphi_M$ is injective after localizing at $\fm$.
\end{proof}

\begin{lemma} \label{lem:iwasawa}
Let $U$ be a dense Zariski open subset of $B$. Then for all $g \in \GL_{\infty}$ the set $\bO_{\infty} U g^{-1} \cap B$ contains a dense Zariski open subset of $B$.
\end{lemma}

\begin{proof}
Pick $g \in \GL_\infty$; then $g \in \GL_n$ for $n$ large enough. Since $\bO_n \cap U_n$ is a finite set, the multiplication map $\bO_n \times U_n \to \GL_n$ has dense image (by a dimension count). Since it is also constructible, it contains a dense open subset which we may assume is closed under multiplication by $\bO_n$. In particular, we conclude that $\bO_{\infty} U g^{-1}$ contains a Zariski dense open subset $V$ such that $\bO_{\infty} V = V$. By a similar argument $\bO_{\infty} B$ contains a dense open subset of $\GL_{\infty}$. This implies that $V \cap \bO_{\infty} B$ is nonempty and hence there exists $h \in \bO_{\infty}$ such that $V \cap h B$ is nonempty. Multiplying on the left by $h^{-1}$ shows that $V \cap B$ is a nonempty open subset of $B$. Since $B$ is a directed union of irreducible spaces, we conclude that $V \cap B$ is a dense open subset of $B$.
\end{proof}

We now begin the proof of fullness. Let $M$ and $N$ be torsion-free $A$-modules and let $\ol{f} \colon M/\fm M \to N/\fm N$ be a map of $\bO_{\infty}$ representations. The diagram in Lemma~\ref{lem:faithful} allows us to define a map $f_{\fm} \colon M_{\fm} \to N_{\fm}$, and this induces a map $f \colon M \otimes_A K \to N \otimes_A K$ a $\vert K \vert$-linear map.  By definition, the map $f_{\fm}$ is characterized as follows: if $x \in M_{\fm}$ and $y \in N_{\fm}$ then $y=f_{\fm}(x)$ if and only if $\ol{f}(\ol{bx})=\ol{by}$ for all $b \in B$, where overlines denote reduction modulo $\fm$. Using Lemma~\ref{lem:open}, we can say more: if $\ol{f}(\ol{bx})=\ol{by}$ for all $b$ in some dense Zariski open subset $U \subset B$ then $y=f_{\fm}(x)$. Indeed, putting $y'=f_{\fm}(x)$ we have $\ol{f}(\ol{bx})=\ol{by'}$ for all $b \in B$, and so $\ol{by}=\ol{by'}$ for all $b \in U$, and so $y=y'$ by the lemma. We now give  a similar characterization for $f$.

\begin{lemma} \label{lem:charf}
Let $x \in M \otimes_A K$ and $y \in N \otimes_A K$. Then $y=f(x)$ if and only if the following condition holds:
\begin{itemize}
\item[$(\ast)$] There exists a dense Zariski open dense subset $U$ of $B$ such that for all $b \in U$ we have $bx \in M_{\fm}$ and $by \in N_{\fm}$ and $\ol{f}(\ol{bx})=\ol{by}$.
\end{itemize}
\end{lemma}

\begin{proof}
Suppose $y=f(x)$. Pick non-zero $a \in A$ such that $ax \in M$. Let $V$ be a dense Zariski open subset of $B$ such that $ba \in A_{\fm}$ and $ba^{-1} \in A_{\fm}$ and $bx \in M_{\fm}$ and $by \in N_{\fm}$ for all $b \in V$ (Lemma~\ref{lem:V}). Put $z=f(ax)$. Since $ax \in M_{\fm}$ we have $z=f_{\fm}(ax)$, and so $\ol{bz}=\ol{f}(\ol{bax})$ for all $b \in B$. For $b \in V$ we have $\ol{f}(\ol{bax})=\ol{ba} \cdot \ol{f}(\ol{bx})$ and $\ol{bz}=\ol{bay}=\ol{ba} \cdot \ol{by}$, and so $\ol{ba} \cdot \ol{by}=\ol{ba} \cdot \ol{f}(\ol{bx})$. Since $ba^{-1} \in A_{\fm}$, it follows that $\ol{ba} \ne 0$, and so $\ol{by}=\ol{f}(\ol{bx})$. So $(\ast)$ holds.

Now suppose $(\ast)$ holds. Let $a$ be a non-zero element of $A$ such that $ax \in M$. Let $z=f(ax)$. Since $ax \in M$, we have $z=f_{\fm}(ax)$, and so $\ol{bz}=\ol{f}(\ol{bax})$ for all $b \in B$. Let $V$ be a dense Zariski open subset of $B$ such that $ba \in A_{\fm}$ for all $b \in V$ (Lemma~\ref{lem:V}). Then for $b \in U \cap V$ we have $\ol{bz}=\ol{f}(\ol{bax})=\ol{ba} \cdot \ol{f}(\ol{bx})=\ol{ba} \cdot \ol{by}=\ol{bay}$. It follows from Lemma~\ref{lem:open} that $z=ay$, and so $ay=f(ax)$. Since $f$ is $K$-linear, we conclude $y=f(x)$.
\end{proof}

\begin{lemma}
The map $f \colon M \otimes_A K \to N \otimes_A K$ is $\GL_{\infty}$-equivariant.
\end{lemma}

\begin{proof}
Let $x \in M \otimes_A K$ and let $y=f(x)$ and let $g \in \GL_{\infty}$. We must show $gy=f(gx)$. Let $U$ be a dense Zariski open subset of $B$ such that $bx \in M_{\fm}$ and $by \in N_{\fm}$ and $\ol{by}=\ol{f}(\ol{bx})$ for all $b \in U$ (Lemma~\ref{lem:charf}). Let $V=\bO_{\infty}Ug^{-1} \cap B$, and let $b \in V$. We can then write $bg=h'b'$ with $h' \in \bO_{\infty}$ and $b' \in U$. We have $bgx=h'b'x \in M_{\fm}$ since $b'x \in M_{\fm}$ and $M_{\fm}$ is stable by $\bO_{\infty}$. Similarly, $bgy \in N_{\fm}$. Furthermore,
\begin{displaymath}
\ol{f}(\ol{bgx})=\ol{f}(\ol{h'b'x})=h' \ol{f}(\ol{b'x})=h' \ol{b'y}=\ol{bgy}.
\end{displaymath}
This is the only place where we use the $\bO_{\infty}$-equivariance of $\ol{f}$. Since this holds for all $b \in V$ and $V$ contains a dense Zariski open of $B$ (Lemma~\ref{lem:iwasawa}), it follows that $gy=f(gx)$ (Lemma~\ref{lem:charf}). This completes the proof.
\end{proof}

We have shown that $\Phi$ is full. The following lemma completes the proof of the theorem.

\begin{lemma}
$\Phi$ is essentially surjective.
\end{lemma}

\begin{proof}
Since $\Phi$ is full and compatible with direct limits, it suffices to show that all finitely generated objects of $\Rep(\bO_{\infty})$ are in the essential image of $\Phi$. Thus let $M$ be such an object. By the results of \cite[\S 4]{infrank}, we can realize $M$ as the kernel of a map $f \colon I \to J$, where $I$ and $J$ are injective objects of $\Rep(\bO_{\infty})$. Every injective object of $\Rep(\bO_{\infty})$ is the restriction to $\bO_{\infty}$ of a polynomial representation of $\GL_{\infty}$. Thus $I=\Phi(M)$ and $J=\Phi(N)$ for some $M$ and $N$ in $\Mod_K$, and $f=\Phi(f')$ for some $f' \colon M \to N$ in $\Mod_K$. The exactness of $\Phi$ shows that $M \cong \Phi(\ker(f'))$, and so $\Phi$ is essentially surjective.
\end{proof}

\subsection{The other two cases} \label{ss:other}

Everything in this section can be adapted to $\Sym(\lw^2(\bC^{\infty}))$. This is straightforward (and not even logically necessary, per Remark~\ref{rmk:transp}), so we do not comment further on it.

Everything can also be adapted to the bivariate tca $A=\Sym(\bC^{\infty} \otimes \bC^{\infty})$. We will make a few comments on how this goes. First, we state the analogs of Theorem~\ref{thm:modK-equiv} and Corollary~\ref{cor:list}. A representation of $\GL_{\infty}$ is {\bf algebraic} if it appears as a subquotient of a (possibly infinite) direct sum of representations of the form $(\bC^{\infty})^{\otimes a} \otimes (\bC^{\infty}_*)^{\otimes b}$. Here $\bC^{\infty}_*$ is the restricted dual of $\bC^{\infty}$, defined as the span of the dual basis $\{e_i^*\}$ in the usual dual space $(\bC^{\infty})^*$. One easily checks that $\bC^{\infty}_*$ is indeed a representation of $\GL_{\infty}$. We write $\Rep(\GL_{\infty})$ for the category of algebraic representations. This was also studied in \cite{infrank}.

By the ``twisted diagonal embedding'' $\GL_{\infty} \to \GL_{\infty} \times \GL_{\infty}$, we mean the embedding given by $g \mapsto (g, {}^tg^{-1})$. We note that the algebraic representations of $\GL_{\infty}$ are exactly those appearing as a subquotient of the restriction of a polynomial representation from $\GL_{\infty} \times \GL_{\infty}$ via the twisted diagonal embedding.

We identify $A$ with $\bC[x_{i,j}]$ in the obvious manner, and let $\fm \subset \vert A \vert$ be the ideal generated by $x_{i,i}-1$ and $x_{i,j}$ for $i \ne j$. This ideal is stable under the twisted diagonal $\GL_{\infty}$. For an $A$-module $M$, define $\wt{\Phi}(M)=M/\fm M$. This is naturally a representation of $\GL_{\infty}$.

\begin{theorem} \label{thm:modK-equiv2}
The functor $\wt{\Phi}$ induces an equivalence $\Phi \colon \Mod_K \to \Rep(\GL_{\infty})$.
\end{theorem}

\begin{corollary} \label{cor:list2}
We have the following:
\begin{enumerate}[\indent \rm (a)]
\item Finitely generated objects of $\Mod_K$ have finite length.
\item If $V$ is a finite length polynomial representation of $\GL_{\infty} \times \GL_{\infty}$ then $K \otimes V$ is a finite length injective object of $\Mod_K$, and all finite length injective objects have this form.
\item Associating to $(\lambda, \mu)$ the socle of $K \otimes \bS_{\lambda}(\bC^{\infty}) \otimes \bS_{\mu}(\bC^{\infty})$ gives a bijection between pairs of partitions and isomorphism classes of simple objects of $\Mod_K$.
\item Every finite length object $M$ of $\Mod_K$ has finite injective resolution $M \to I_{\bullet}$ where each $I_k$ is a finite length injective object.
\end{enumerate}
\end{corollary}

\begin{proof}
These properties are proven for $\Rep(\GL_{\infty})$ in \cite{infrank}:

(a) \cite[Proposition 3.1.5]{infrank}, 

(b,c) \cite[Proposition 3.2.14]{infrank}, 

(d) dualize the explicit projective resolutions in \cite[(3.3.7)]{infrank}. \qedhere
\end{proof}

The proof of Theorem~\ref{thm:modK-equiv2} closely follows that of Theorem~\ref{thm:modK-equiv}. The main differences occur in the analog of \S \ref{ss:struc-at-m}. In the present case, one takes $\ol{B} \subset \rM_{\infty} \times \rM_{\infty}$ to be the set of pairs of upper-triangular matrices. The group $H$ is replaced with the intersection of $\ol{B}$ and the twisted diagonal $\GL_{\infty}$ inside of $\GL_{\infty} \times \GL_{\infty}$, and consists of pairs $(h, h^{-1})$ where $h \in \GL_{\infty}$ is a diagonal matrix. With these definitions, everything proceeds in a similar way.

\section{Proof of the main theorems} \label{sec:proof}

\subsection{The structure of ideals}
\label{sec:ideal}

We have the following multiplicity-free decompositions:
\begin{align*}
\Sym(\Sym^2 \bC^\infty) &= \bigoplus \bS_{2\lambda}(\bC^\infty)\\
\Sym(\lw^2 \bC^\infty) &= \bigoplus \bS_{(2\lambda)^\dagger}(\bC^\infty)\\
\Sym(\bC^\infty \otimes \bC^\infty) &= \bigoplus \bS_{\lambda}(\bC^\infty) \otimes \bS_\lambda(\bC^\infty).
\end{align*}
For a proof, see \cite[\S I.5, Example 5]{macdonald} for the first two decompositions and \cite[\S I.4, (4.3)]{macdonald} for the last one. In all cases, the sum is over partitions $\lambda$. For the purposes of stating the next result we write $E_{\lambda}$ for the $\lambda$ summand. Let $I_\lambda$ be the ideal generated by $E_\lambda$.

\begin{proposition} \label{prop:ideal}
$E_\mu \subseteq I_\lambda$ if and only if $\lambda \subseteq \mu$.
\end{proposition}

\begin{proof}
For $\Sym(\Sym^2 \bC^\infty)$, see \cite{abeasis}, for $\Sym(\bigwedge^2 \bC^\infty)$, see  \cite[Theorem 3.1]{pfaffians}, and for $\Sym(\bC^\infty \otimes \bC^\infty)$, see \cite[Theorem 4.1]{CEP}. Since \cite{abeasis} is a difficult reference to obtain, we note that the result for $\Sym(\Sym^2 \bC^\infty)$ follows from that of $\Sym(\bigwedge^2 \bC^\infty)$ because the two are transpose dual (see \cite[\S 7.4]{expos}). Proofs of these results will also appear in  \cite{NSS}.
\end{proof}

\begin{corollary} \label{cor:quo-noeth}
Let $A$ be one of the three algebras above, and let $I$ be any non-zero ideal of $A$. Then $A/I$ is essentially bounded, and, in particular, noetherian.
\end{corollary}

\begin{proof}
Suppose that $I$ is a non-zero ideal of $A$. Then $I$ contains some $E_{\lambda}$, and thus $I_{\lambda}$. Thus by the proposition, $A/I$ contains no partition $\mu$ satisfying $\lambda \subset \mu$, and is therefore essentially bounded. Noetherianity of $A/I$ follows from Proposition~\ref{prop:eb}.
\end{proof}

\subsection{The (FT) property} \label{ss:ft}

Let $B$ be a (bivariate) tca with $B_0=\bC$, so that $B_+$ (the ideal of $B$ generated by positive degree elements) is maximal. We say that a $B$-module $M$ is {\bf (FT) over $B$} if $\Tor^B_i(M, \bC)$ is a finite length representation of $\GL_{\infty}$ (or $\GL_{\infty} \times \GL_{\infty}$) for all $i \ge 0$. The $i=0$ case implies that $M$ is finitely generated as a $B$-module, by Nakayama's lemma \cite[Proposition 8.4.2]{expos}. Conversely, if $B$ is noetherian then any finitely generated $B$-module satisfies (FT). We note that if
\begin{displaymath}
0 \to M_1 \to M_2 \to M_3 \to 0
\end{displaymath}
is a short exact sequence of $B$-modules and two of the modules are (FT) then so is the third.

The main result we need concerning (FT) is the following proposition:

\begin{proposition} \label{prop:tors-FT}
Let $A$ be one of $\Sym(\Sym^2 \bC^\infty)$, $\Sym(\bigwedge^2 \bC^\infty)$, or $\Sym(\bC^\infty \otimes \bC^\infty)$, and let $M$ be a finitely generated torsion $A$-module. Then $M$ satisfies (FT) over $A$.
\end{proposition}

We begin with some lemmas.

\begin{lemma}
\label{lem:FT}
Let $B \to B'$ be a homomorphism of (bivariate) tca's, with $B_0=B'_0=\bC$, and let $M$ be a $B'$-module. Suppose that $B'$ is (FT) over $B$. Then $M$ is (FT) over $B$ if and only if $M$ is (FT) over $B'$.
\end{lemma}

\begin{proof}
First suppose that $M$ is (FT) over $B'$. Starting with a free resolution of $M$ over $B'$, and a free resolution of $B'$ over $B$, we get an acyclic double complex of $B$-modules resolving $M$. This leads to a spectral sequence
\begin{displaymath}
\rE^2_{p,q} = \Tor^B_p(\Tor^{B'}_q(M, \bC), B') \implies \Tor^B_{p+q}(M, \bC).
\end{displaymath}
Note that $\Tor^{B'}_q(M, \bC)$ is a trivial $B$ module (meaning $B_+$ acts by 0), and so
\begin{displaymath}
\Tor^B_p(\Tor^{B'}_q(M, \bC), B') = \Tor^{B'}_q(M, \bC) \otimes_{\bC} \Tor^B_p(B', \bC).
\end{displaymath}
Each of the $\Tor$'s on the right has finite length by assumption, and so the left side also has finite length. It follows that $\Tor^B_{i+j}(M, \bC)$ has finite length, and so $M$ is (FT) over $B$.

Now suppose that $M$ is (FT) over $B$. In particular, $M$ is a finitely generated $B$-module, and so also a finitely generated $B'$-module. This shows that $\Tor_0^{B'}(M,\bC)$ is finite length. Let $P \to M \to 0$ be a minimal projective cover and let $N$ be the kernel. Since $B'$ is (FT) over $B$, we conclude that $P$, and hence $N$ are both (FT) over $B$. In particular, $N$ is a finitely generated as a module over $B$, and hence over $B'$. This shows that $\Tor_1^{B'}(M,\bC)$ is finite length; to get the statement for $\Tor_i^{B'}(M,\bC)$, we can iterate this argument $i$ times.
\end{proof}

\begin{lemma} \label{lem:symc11noeth}
Let $A$ be the bivariate tca $\Sym(\bC^{\infty} \otimes \bC^{\infty})$. Then $A/I_\lambda$ satisfies (FT) over $A$ for all rectangular partitions $\lambda$.
\end{lemma}

\begin{proof}
This follows from \cite[Theorem 1.2]{raicu}, taking $m=n=\infty$ (the results there are stated for finite $m$ and $n$, but since the answer is given in terms of Schur functors, it can be extended to the infinite case): one has to show that the coefficient of $w^i$, as a polynomial in $z$, is of bounded degree. To see that, note that fixing $w^i$ means that $q$ is bounded from above, and then the result is clear from the form of the polynomials $h_{r \times s}(z,w)$.
\end{proof}

\begin{lemma} \label{lem:deg2-FT}
Let $B$ be $\Sym(\Sym^2 \bC^\infty)$ or $\Sym(\bigwedge^2 \bC^\infty)$. Then $B/I_\lambda$ satisfies (FT) over $B$ for all rectangular partitions $\lambda$. 
\end{lemma}

\begin{proof}
Let $A = \Sym(\bC^\infty \otimes \bC^\infty)$. Let $J_\lambda$ be the ideal in $A$ generated by $\bS_\lambda \otimes \bS_\lambda$. Let $\wt{A}$ be the tca obtained from $A$ by restricting to the diagonal $\GL_\infty$ action. Then there is a surjection of tca's $\phi \colon \wt{A} \to B$, induced by the natural map $(\bC^{\infty})^{\otimes 2} \to \Sym^2(\bC^{\infty})$, and $\phi(J_\lambda) \subset I_\lambda$. (Note that $\phi(J_\lambda)$ is nonzero: if $\lambda$ is a single column, then this is an ideal generated by minors of a given size and the image of every power of $J_\lambda$ is nonzero; in general, some power of a determinantal ideal belongs to $J_\lambda$ after we specialize to large enough finite-dimensional vector spaces.)

Since $A/J_\lambda$ is (FT) over $A$ (Lemma~\ref{lem:symc11noeth}), each $\Tor_i^A(A/J_\lambda, \bC)$ is a finite length $\GL_\infty \times \GL_\infty$ module, and hence remains finite length under the restriction to the diagonal copy of $\GL_\infty$. So $\wt{A}/J_\lambda$ is (FT) over $\wt{A}$. Also $\wt{A}/J_\lambda$ is essentially bounded (since the bivariate tca $A/J_{\lambda}$ is) and hence noetherian (Proposition~\ref{prop:eb}). It follows that $B/\phi(J_\lambda)$ is (FT) over $\wt{A}/J_\lambda$, thus over $\wt{A}$ as well (Lemma~\ref{lem:FT}).

Next, $B$ is (FT) over $\wt{A}$ (the resolution of $B$ over $\wt{A}$ is a Koszul complex) so another application of Lemma~\ref{lem:FT} gives that $B/\phi(J_\lambda)$ is (FT) over $B$. Finally, $B/I_\lambda$ is a finitely generated module over $B/\phi(J_\lambda)$ and the latter is noetherian (Corollary~\ref{cor:quo-noeth}), so $B/I_\lambda$ is (FT) over $B/\phi(J_\lambda)$. We apply Lemma~\ref{lem:FT} again to deduce that $B/I_\lambda$ is (FT) over $B$.
\end{proof}

\begin{remark}
It would be interesting to prove directly that $B/I_\lambda$ satisfies (FT) over $B$ by computing $\Tor^B_i(B/I_\lambda, \bC)$, as is done in \cite{raicu} for $\Sym(\bC^\infty \otimes \bC^\infty)$.
\end{remark}

\begin{proof}[Proof of Proposition~\ref{prop:tors-FT}]
Let $I$ be the annihilator of $M$. This is non-zero by Corollary~\ref{cor:nz-ann}. Thus $I$ contains an ideal generated by a rectangular partition; replace $I$ with this ideal. Since $A/I$ is noetherian (Corollary~\ref{cor:quo-noeth}), $M$ is (FT) over $A/I$. By Lemma~\ref{lem:symc11noeth} or~\ref{lem:deg2-FT}, $A/I$ is (FT) over $A$. Thus by Lemma~\ref{lem:FT}, $M$ is (FT) over $A$.
\end{proof}

\subsection{Completion of the proofs}

Let $A$ be one of the tca's $\Sym(\Sym^2 \bC^\infty)$ or $\Sym(\bigwedge^2 \bC^\infty)$, or the bivariate tca $\Sym(\bC^\infty \otimes \bC^\infty)$, and let $K=\Frac(A)$.

\begin{proposition} \label{prop:S-fin}
If $M$ is a finite length $K$-module then $S(M)$ satisfies (FT) over $A$.
\end{proposition}

\begin{proof}
We prove this by induction on the injective dimension of $M$, which is possible by Corollary~\ref{cor:list}(d) (and its analogs). If $M$ is injective then $S(M)$ is a finitely generated projective $A$-module (Corollary~\ref{cor:list}(b), Proposition~\ref{prop:poly-sat}), and thus satisfies (FT). Now let $M$ be a finite length object of $\Mod_K$ with positive injective dimension. We can then find an exact sequence
\begin{displaymath}
0 \to M \to I \to N \to 0,
\end{displaymath}
where $I$ is injective and $N$ has smaller injective dimension than $M$. Applying $S$, we obtain an exact sequence
\begin{displaymath}
0 \to S(M) \to S(I) \to S(N) \to (\rR^1 S)(M) \to 0.
\end{displaymath}
By induction, $S(N)$ is (FT) over $A$, and so finitely generated. It follows that $(\rR^1 S)(M)$ is finitely generated. By Proposition~\ref{prop:FTS-id} and the fact that  localization is exact, we have $(\rR^1 S)(M) \otimes_A K=0$, and so $(\rR^1 S)(M)$ satisfies (FT) over $A$ by Proposition~\ref{prop:tors-FT}. Thus $S(I)$, $S(N)$, and $(\rR^1 S)(M)$ all satisfy (FT) over $A$, and so $S(M)$ satisfies (FT) over $A$ as well.
\end{proof}

The following completes the proof of our main results: Theorem~\ref{mainthm} and Theorem~\ref{mainthm2}.

\begin{theorem} \label{thm:A-noeth}
$A$ is noetherian.
\end{theorem}

\begin{proof}
Let $P$ be a finitely generated projective $A$-module, and let $N_1 \subset N_2 \subset \cdots$ be an ascending chain of $A$-submodules of $P$. Since $P \otimes_A K$ is finite length (Corollary~\ref{cor:list}(a)), it follows that the ascending chain $N_i \otimes_A K$ stabilizes, and so we may as well assume it is stationary to begin with. Let $M \subset P$ be the common value of $S(N_i \otimes_A K)$, which is finitely generated by Proposition~\ref{prop:S-fin}. Then $N_{\bullet}$ is an ascending chain in $M$. Let $M'=M/N_1$ and $N'_i=N_i/N_1 \subset M'$, so that $N'_{\bullet}$ is an ascending chain in $M'$. Since $M'$ is finitely generated and $M' \otimes K=0$, Corollary~\ref{cor:nz-ann} implies that $I=\ann(M')$ is non-zero. Thus $M$ is a module over $A/I$, which is noetherian (Corollary~\ref{cor:quo-noeth}), and so $N'_{\bullet}$ stabilizes. This implies that $N_{\bullet}$ stabilizes, and so $P$ is noetherian.
\end{proof}

\begin{remark}
The above proof has three key ingredients: 
\begin{enumerate}[\indent (1)]
\item Finitely generated objects of $\Mod_K$ are noetherian.
\item If $I$ is a non-zero ideal of $A$ then $A/I$ is noetherian.
\item If $M$ is a finite length object of $\Mod_K$ then $S(M)$ is a finitely generated $A$-module. 
\end{enumerate}
Let us make one comment regarding (3). Given a finite length object $M$ in $\Mod_K$, we can realize $M$ as the kernel of a map $I \to J$ where $I$ and $J$ are finite length injective objects of $\Mod_K$. Since $S$ is left-exact, it follows that $S(M)$ is the kernel of the map $S(I) \to S(J)$, and we know that $S(I)$ and $S(J)$ are finitely generated projective $A$-modules. Thus finite generation of $S(M)$ would follow immediately if we knew $A$ to be \emph{coherent} (which exactly says that the kernel of a map of finitely generated projective modules is finitely generated). Since coherence is a weaker property than noetherianity, it should be easier to prove; however, we have not found any way to directly prove coherence.
\end{remark}

\section{A Gr\"obner-theoretic approach to the main theorems}
\label{sec:grobner}

In this section we outline a possible approach to proving Theorem~\ref{mainthm} using Gr\"obner bases. This leads to an interesting combinatorial problem that we do not know how to resolve.

\subsection{Admissible weights}

A {\bf weight} of $\GL_{\infty}$ is a sequence of non-negative integers $w=(w_1, w_2, \ldots)$ such that $w_i=0$ for $i \gg 0$. Every polynomial representation $V$ of $\GL_{\infty}$ decomposes as $V=\bigoplus V_w$, where $V_w$ is the $w$ weight space. A weight is {\bf admissible} if $w_i$ is 0 or 1 for all $i$. An {\bf admissible weight vector} is an element of some $V_w$ with $w$ an admissible weight. We require the following fact: if $V$ is a polynomial representation of $\GL_{\infty}$ then $V$ is generated, as a representation, by its admissible weight vectors.

\subsection{Degree one tca's}

We begin by sketching a Gr\"obner-theoretic proof that the tca $A=\Sym(\bC^{\infty} \oplus \bC^{\infty})$ is noetherian. This proof comes from transferring the proof in \cite{catgb} that $\Rep(\FI_2)$ is noetherian through Schur--Weyl duality, and can easily be adapted to treat all tca's generated in degree $\le 1$. Let $x_1, x_2, \ldots$ be a basis for the first $\bC^{\infty}$, and let $y_1, y_2, \ldots$ be a basis for the second $\bC^{\infty}$, so that $A$ is the polynomial ring $\bC[x_1, x_2, \ldots, y_1, y_2, \ldots]$.

Let $\cM$ be the set of pairs $\Gamma=(S, \phi)$, where $S$ is a finite subset of $\bN=\{1, 2, \ldots\}$ and $\phi \colon S \to \{ \text{red}, \text{blue} \}$ is a function. Given $\Gamma, \Gamma' \in \cM$, we define $\Gamma \to \Gamma'$ (a ``move'') if one of the following two conditions hold:
\begin{itemize}
\item $S'$ is obtained from $S$ by adding a single element and leaving the colors unchanged (i.e., $\phi' \vert_S = \phi$).
\item There exists some $i \in S$ such that $i+1 \not\in S$ and $S'$ is obtained from $S$ by replacing $i$ with $i+1$ (and leaving all colors unchanged).
\end{itemize}
We define $\Gamma \le \Gamma'$ if there is a sequence of moves taking $\Gamma$ to $\Gamma'$. This partially orders $\cM$.

We also define a total order $\preceq$ on $\cM$, as follows. Given two finite subsets $S$ and $S'$ of $\bN$, define $S \preceq S'$ if $\max(S)<\max(S')$, or $\max(S)=\max(S')=n$ and $S \setminus \{n\} \preceq S' \setminus \{n\}$. Given $S \subset \bN$ and $\phi, \phi' \colon S \to \{ \text{red}, \text{blue} \}$, define $\phi \preceq \phi'$ by thinking of $\phi$ and $\phi'$ as words in R and B and using the lexicographic order (with $R \preceq B$, say). Finally, define $(S,\phi) \preceq (S',\phi')$ using the lexicographic order (i.e., $S \prec S'$, or $S=S'$ and $\phi \preceq \phi'$).

Given $\Gamma \in \cM$, define
\begin{displaymath}
m_{\Gamma} = \prod_{i \in S} \begin{cases} x_i & \text{if $\phi(i)=$ red} \\ y_i & \text{if $\phi(i)=$ blue} \end{cases}.
\end{displaymath}
If $f \in A$ is an admissible weight vector of weight $w$, then $f$ is a linear combination of the $m_{\Gamma}$'s where $\Gamma$ has the same support as $w$. We define the initial variable of $f$, denoted $\inn(f)$, to be the largest $\Gamma$ (under $\preceq$) such that $m_{\Gamma}$ appears in $f$ with non-zero coefficient.

Now let $I$ be an ideal of $A$. Let $\inn(I) \subset \cM$ be the set of $\inn(f)$'s where $f$ varies over the admissible weight vectors in $I$. One then proves the following two statements: 
\begin{enumerate}[\indent (1)]
\item $\inn(I)$ is a poset ideal of $\cM$, that is, $\inn(I)$ is closed under moves, and 
\item if $I \subset J$ and $\inn(I)=\inn(J)$ then $I=J$. 
\end{enumerate}
From this, weak noetherianity of $A$ follows from noetherianity of $\cM$, which is an easy exercise. A slight modification of this argument shows that $A$ is noetherian.

\subsection{Degree two tca's}

We now sketch our Gr\"obner approach to the noetherianity of $A=\Sym(\Sym^2(\bC^{\infty}))$. Let $x_{i,j}$, with $i \le j$, be a basis for $\Sym^2(\bC^{\infty})$, so that $A=\bC[x_{i,j}]$.

Let $\cM$ be the set of undirected matchings $\Gamma$ on $\bN$. (Recall that a graph is a matching if each vertex has valence 0 or 1.) Given $\Gamma, \Gamma' \in \cM$, we define $\Gamma \to \Gamma'$ if one of the following two conditions hold:
\begin{itemize}
\item $\Gamma'$ is obtained from $\Gamma$ by adding a single edge.
\item There exists an edge $(i,j)$ in $\Gamma$ such that $j+1$ is not in $\Gamma$, and $\Gamma'$ is obtained from $\Gamma$ by replacing $(i,j)$ with $(i,j+1)$. (Here we allow $i<j$ or $j<i$.)
\end{itemize}
We call $\Gamma \to \Gamma'$ a ``type~I move''. We define $\Gamma \le \Gamma'$ if there is a sequence of type~I moves transforming $\Gamma$ to $\Gamma'$. This partially orders $\cM$.

We also define a total order $\preceq$ on $\cM$ as follows. First, suppose that $i<j$ and $k<\ell$ are elements of $\bN$. Define $(i,j) \preceq (k,\ell)$ if $j<\ell$, or $j = \ell$ and $i \le k$. Now, let $\Gamma$ and $\Gamma'$ be two elements of $\cM$, and let $e_1 \preceq \cdots \preceq e_n$ and $e'_1 \preceq \cdots \preceq e'_m$ be their edges, listed in increasing order. We define $\Gamma \preceq \Gamma'$ if $m > n$, or if $m=n$ and $(e_1, \ldots, e_n) \preceq (e'_1, \ldots, e'_m)$ under the lexicographic order. 

Given $\Gamma \in \cM$, define $m_{\Gamma} = \prod_{(i,j) \in \Gamma} x_{i,j}$. Once again, every admissible weight vector is a sum of $m_{\Gamma}$'s, and we define the initial term $\inn(f)$ of an admissible weight vector $f$ to be the largest $\Gamma$ (under the order $\preceq$) for which the coefficient of $m_{\Gamma}$ is non-zero in $f$.

Let $I$ be an ideal of $A$. Define $\inn(I)$ as before. Once again, $\inn(I)$ is closed under type~I moves, and therefore forms a poset ideal of $(\cM, \le)$. The weak noetherianity of $A$ would follow from the noetherianity of the poset $(\cM, \le)$, but the latter property fails:

\begin{example} \label{eg:not-noeth}
For $n \ge 3$, define $\Gamma_n \in \cM$ to have edges $(2i+1, 2i+4)$ for $i=0,1,\dots,n-2$ and $(2,2n-1)$. Then $\Gamma_n$ is supported on $\{1,\dots,2n\}$. It is easy to verify that the $\Gamma_n$ are incomparable, so $(\cM, \le)$ is not a noetherian poset.
\end{example}

The above observation is not the end of the road, however: the set $\inn(I)$ is closed under more than just type~I moves. Suppose $\Gamma \in \inn(I)$ and that $e=(i, j)$ and $e'=(k, \ell)$ are edges appearing in $\Gamma$, with $i<j$ and $k<\ell$ and $j<\ell$. We then have the following observations: 
\begin{itemize}
\item Suppose $k<i<j<\ell$ and that every number strictly between $k$ and $i$ that appears in $\Gamma$ is connected to a number larger than $j$. Let $\Gamma'$ be the graph obtained by replacing $e$ and $e'$ with $(k, j)$ and $(i, \ell)$. Then $\Gamma' \in \inn(I)$.
\item Suppose $i<k<j<\ell$ and that every number strictly between $k$ and $j$ that appears in $\Gamma$ is connected to a number larger than $j$. Let $\Gamma'$ be the graph obtained by replacing $e$ and $e'$ with $(i, k)$ and $(j, \ell)$. Then $\Gamma' \in \inn(I)$.
\end{itemize}
Write $\Gamma \Rightarrow \Gamma'$ to indicate that $\Gamma'$ is related to $\Gamma$ by one of the above two modifications. We call this a ``type~II move.'' Here is a pictorial representation of these moves (we use labels $a<b<c<d$, and the dotted lines indicate that any element there is either not on an edge, or is connected to a number larger than $c$):
\begin{center}
  \begin{tikzpicture}[baseline=4pt,thick,font=\small]
    \path (0,0) node (a) {$a$}
          (1,0) node (b) {$b$}
          (2,0) node (c) {$c$}
          (3,0) node (d) {$d$};
    \draw (a) to[out=45] (d);
    \draw[dotted] (a) to (b);
    \draw (b) to[out=45] (c);
  \end{tikzpicture}
\;$\Rightarrow$\;
  \begin{tikzpicture}[baseline=4pt,thick,font=\small]
    \path (0,0) node (a) {$a$}
          (1,0) node (b) {$b$}
          (2,0) node (c) {$c$}
          (3,0) node (d) {$d$};
    \draw (a) to[out=45] (c);
    \draw[dotted] (b) to (c);
    \draw (b) to[out=45] (d);
  \end{tikzpicture}
\;$\Rightarrow$\;
  \begin{tikzpicture}[baseline=4pt,thick,font=\small]
    \path (0,0) node (a) {$a$}
          (1,0) node (b) {$b$}
          (2,0) node (c) {$c$}
          (3,0) node (d) {$d$};
    \draw (a) to[out=45] (b);
    \draw (c) to[out=45] (d);
  \end{tikzpicture}
\end{center}
We define a new partial order $\sqsubseteq$ on $\cM$ as follows: $\Gamma \sqsubseteq \Gamma'$ if there exists a sequence of moves (of any type) taking $\Gamma$ to $\Gamma'$. The above observations show that $\inn(I)$ is a poset ideal of $(\cM, \sqsubseteq)$. This leads to the important open question:

\begin{question}
\label{ques:grobner}
Is the poset $(\cM, \sqsubseteq)$ noetherian?
\end{question}

\begin{remark}
The sequence defined in Example~\ref{eg:not-noeth} \emph{is} comparable in $(\cM, \sqsubseteq)$. Let $\sigma_i$ be the element $(i, i+1)\cdots(3,4)(2,3)$ of the symmetric group $S_{2 n}$. For each $2 \leq i \leq 2 n -4$, we have type II moves $\sigma_i \Gamma_n \to \sigma_{i+1} \Gamma_n$, so $\Gamma_n \sqsubseteq \sigma_{2 n - 3}\Gamma_n$. Finally, $(2n-1, 2n)$ is a valid type~II move for $\sigma_{2 n - 3}\Gamma_n$. It is now easy to check that $((2n-1, 2n) \sigma_{2 n -3}) \Gamma_n$ embeds into $\Gamma_{m}$ (via type~I moves) for any $m>n$. This shows $\Gamma_n \sqsubseteq \Gamma_m$ for any $m>n \ge 3$.
\end{remark}

A positive answer to Question~\ref{ques:grobner} would show that $A$ is weakly noetherian. A slight modification of this question would give noetherianity. Furthermore, this approach would even give results in positive characteristic.


\begin{thebibliography}{CEFN}

\bibitem[Ab]{abeasis} Silvana Abeasis, The ${\rm GL}(V)$-invariant ideals in $S(S^{2}V)$, {\it Rend. Mat. (6)} {\bf 13} (1980), no.~2, 235--262.

\bibitem[AdF]{pfaffians} S.~Abeasis, A.~Del Fra, Young diagrams and ideals of Pfaffians, {\it Adv. in Math.} {\bf 35} (1980), no.~2, 158--178.

\bibitem[AH]{aschenbrennerhillar}
Matthias Aschenbrenner, Christopher J. Hillar,  Finite generation of symmetric ideals, {\it Trans.\ Amer.\ Math.\ Soc.}\ {\bf 359} (2007), 5171--5192, \arxiv{math/0411514v3}.

\bibitem[BR]{BR} A.~Berele, A.~Regev, Hook Young diagrams with applications to combinatorics and to representations of Lie superalgebras, {\it Adv. in Math.} {\bf 64} (1987), 118--175.

\bibitem[CEF]{fimodules} Thomas Church, Jordan Ellenberg, Benson Farb, FI-modules and stability for representations of symmetric groups, {\it Duke Math. J.} {\bf 164} (2015), no.~9, 1833--1910, \arxiv{1204.4533v4}.

\bibitem[CEFN]{fi-noeth} Thomas Church, Jordan S. Ellenberg, Benson Farb, Rohit Nagpal, FI-modules over Noetherian rings, {\it Geom. Top.} {\bf 18} (2014), 2951--2984, \arxiv{1210.1854v2}.

\bibitem[Co]{cohen} D.~E. Cohen, On the laws of a metabelian variety, {\it J. Algebra} {\bf 5} (1967), 267--273.

\bibitem[CEP]{CEP} C.~de Concini, David Eisenbud, C.~Procesi, Young diagrams and determinantal varieties, {\it Invent. Math.} {\bf 56} (1980), no.~2, 129--165. 

\bibitem[Dr]{draisma-notes} Jan Draisma, Noetherianity up to symmetry, {\it Combinatorial algebraic geometry}, Lecture Notes in Math. {\bf 2108}, Springer, 2014, \arxiv{1310.1705v2}.

\bibitem[DE]{draismaeggermont} Jan Draisma, Rob H. Eggermont, Pl\"ucker varieties and higher secants of Sato's Grassmannian, {\it J. Reine Angew. Math.}, to appear, \arxiv{1402.1667v3}.

\bibitem[DK]{draismakuttler} Jan Draisma, Jochen Kuttler, Bounded-rank tensors are defined in bounded degree, {\it Duke Math. J.} {\bf 163} (2014), no.~1, 35--63, \arxiv{1103.5336v2}.

\bibitem[Eg]{eggermont}
Rob~H. Eggermont, Finiteness properties of congruence classes of infinite-by-infinite matrices, {\it Linear Algebra Appl.} {\bf 484} (2015), 290--303, \arxiv{1411.0526v1}.

\bibitem[HS]{hillarsullivant} Christopher J. Hillar, Seth Sullivant, Finite Gr\"obner bases in infinite dimensional polynomial rings and applications, {\it Adv. Math.} {\bf 221} (2012), 1--25, \arxiv{0908.1777v2}.

\bibitem[L]{lascoux} Alain Lascoux, Syzygies des vari\'et\'es d\'eterminantales, {\it Adv. in Math.} {\bf 30} (1978), no.~3, 202--237.

\bibitem[M]{macdonald} I. G.~Macdonald, {\it Symmetric Functions and Hall Polynomials}, second edition, Oxford Mathematical Monographs,
Oxford, 1995.

\bibitem[NSS]{NSS} Rohit Nagpal, Steven V Sam, Andrew Snowden, Noetherianity of some degree two twisted skew-commutative algebras; in preparation.

\bibitem[RW]{raicu}
Claudiu Raicu, Jerzy Weyman, The syzygies of some thickenings of determinantal varieties, \arxiv{1411.0151v1}.

\bibitem[PS]{putnamsam}
Andrew Putnam, Steven V Sam, Representation stability and finite linear groups, \arxiv{1408.3694v2}.

\bibitem[SS1]{symc1}
Steven V Sam, Andrew Snowden, GL-equivariant modules over polynomial rings in infinitely many variables, {\it Trans. Amer. Math. Soc.}, to appear, \arxiv{1206.2233v3}.

\bibitem[SS2]{expos} Steven~V Sam, Andrew Snowden, Introduction to twisted commutative algebras, \arxiv{1209.5122v1}.

\bibitem[SS3]{infrank}
Steven V Sam, Andrew Snowden, Stability patterns in representation theory, {\it Forum Math., Sigma} {\bf 3} (2015), e11, 108 pp., \arxiv{1302.5859v2}.

\bibitem[SS4]{catgb}
Steven V Sam, Andrew Snowden, Gr\"obner methods for representations of combinatorial categories, \arxiv{1409.1670v2}.

\bibitem[Sn]{delta-mod} Andrew Snowden, Syzygies of Segre embeddings and $\Delta$-modules, {\it Duke Math.\ J.} {\bf 162} (2013), no.~2, 225--277, \arxiv{1006.5248v4}.

\bibitem[We]{weyman} Jerzy Weyman, {\it Cohomology of Vector Bundles and Syzygies}, Cambridge Tracts in Mathematics {\bf 149}, Cambridge University Press, Cambridge, 2003.

\end{thebibliography}
\end{document}